\documentclass[a4paper,twoside]{amsart}
\usepackage{url}
\usepackage{a4wide}
\usepackage{amsthm}
\usepackage{amsmath}
\usepackage{amssymb}
\usepackage{enumitem}
\usepackage{mathtools}
\def\Forb{\mathop{\mathrm{Forb_{e}}}\nolimits}

\theoremstyle{definition}
\newtheorem{definition}{Definition}[section]

\newtheorem*{remark*}{Remark}
\newtheorem*{claim*}{Claim}
\theoremstyle{remark}

\theoremstyle{plain}
\newtheorem{theorem}{Theorem}[section]

\newtheorem{prop}[theorem]{Proposition}
\newtheorem{observation}[theorem]{Observation}

\newtheorem{problem}[theorem]{Problem}

\def\str#1{\mathbf {#1}}

\def\rel#1#2{R_{\mathbf{#1}}^{#2}}
\def\nbrel#1#2{R_{#1}^{#2}}

\def\rel#1#2{R_{\mathbf{#1}}^{#2}}
\def\func#1#2{f_{\mathbf{#1}}^{#2}}

\begin{document}
\title{Type-respecting amalgamation and big Ramsey degrees}
\authors{
	\author[A.~Aranda et al.]{Andr\'es Aranda}
	\author[]{Saumuel Braunfeld}
	\author[]{David Chodounsk\'y}
	\author[]{Jan Hubi\v cka}
	\author[]{Mat\v ej Kone\v cn\'y}
	\address{Department of Applied Mathematics (KAM)\\ Charles University\\ Prague, Czech Republic}
	\email{\{aranda|braunfeld|chodounsky|hubicka|matej\}@kam.mff.cuni.cz}
	\author[]{Jaroslav Ne\v set\v ril}
	\address{Computer Science Institute (IUUK)\\ Charles University\\ Prague, Czech Republic}
	\email{nesetril@iuuk.mff.cuni.cz}
	\author[]{Andy Zucker}
	\address{Department of Pure Mathematics\\ University of Waterloo\\ Canada}
	\email{a3zucker@uwaterloo.ca}
}
\maketitle              
\begin{abstract}
	\vskip-1cm
	We give an infinitary extension of the Ne\v set\v ril--R\"odl theorem for category of relational structures with  special type-respecting embeddings.
\end{abstract}
\section{Introduction}
We use the standard model-theoretic notion of structures allowing functions to be partial.
Let $L$ be a \emph{language} with relational symbols $\rel{}{}\in L$ and functional symbols $f\in L$ each having its {\em arity}.
An \emph{$L$-structure} $\str{A}$ on $A$ is a structure with {\em vertex set} $A$, relations $\rel{A}{}\subseteq A^r$ for every relation symbol $\rel{}{}\in L$ of arity $r$ and partial functions $\func{A}{}\colon A^s\to A$ for every function symbol $\func{}{}\in L$ of arity $s$.  If the set $A$ is finite say that $\str A$ is \emph{finite} (it may still have infinitely many relations if $L$ is infinite). We consider only $L$-structures with finitely many or countably infinitely many vertices.
Language $L$ is \emph{relational} if it contains no function symbols. We say that $\str{A}$ is a \emph{substructure} of $\str{B}$ and write $\str{A}\subseteq \str{B}$ if the identity map is an embedding $\str{A}\to\str{B}$.
Let $\mathcal K$ be a class of $L$-structures.  We say that $\mathcal K$ is \emph{hereditary} if it is closed for substructures.
We say that $L$-structure $\str{U}\in \mathcal K$ is \emph{$\mathcal K$-universal} if every $L$-structure in $\mathcal K$ embeds to $\str{U}$.

Given $L$-structures $\str{A}$ and $\str{B}$, we denote by $\str{B}\choose \str{A}$ the set
of all embeddings from $\str{A}$ to $\str{B}$. We write $\str{C}\longrightarrow (\str{B})^\str{A}_{k,l}$ to denote the following statement:
for every colouring $\chi$ of $\str{C}\choose\str{A}$ with $k$ colours, there exists
an embedding $f\colon\str{B}\to\str{C}$ such that $\chi$ does not take more than $l$ values on $f(\str{B})\choose \str{A}$.
For a countably infinite $L$-structure $\str{B}$ and its finite substructure $\str{A}$, the \emph{big Ramsey degree} of $\str{A}$ in $\str{B}$ is
the least number $D\in \mathbb N\cup \{\infty\}$ such that $\str{B}\longrightarrow (\str{B})^\str{A}_{k,D}$ for every $k\in \mathbb N$.
We say that $L$-structure $\str{B}$ has \emph{finite big Ramsey degrees} if the big Ramsey degree of every finite substructure $\str{A}$ of $\str{B}$ is finite.

In general, we are interested in the following question: Given a hereditary class of $L$-structures $\mathcal K$, do $\mathcal K$-universal $L$-structures in $\mathcal K$ have finite big Ramsey degrees? Notice that if one $\mathcal K$-universal $L$-structure in $\mathcal K$ has finite big Ramsey degrees then all of them do.
The study of big Ramsey degrees originated in 1960's Laver's unpublished proof that the big Ramsey degrees of the
order of rationals are finite. This result was refined and a precise formula was obtained by Devlin~\cite{devlin1979}. This area has recently been revitalized with a rapid
progress regarding big Ramsey degrees of structures in finite binary languages
(see e.g. recent survey~\cite{dobrinen2021ramsey}).

We call $L$-structure $\str{A}$ \emph{irreducible} if for every pair of vertices $u,v\in A$ there exist
a relational symbol $R\in L$ and a tuple $\vec{t}\in \rel{A}{}$ such that $u,v\in \vec{t}$ (this is an easy generalization of the
notion of graph clique).  Given set of $L$-structures $\mathcal F$, $L$-structure $\str{A}$ is \emph{$\mathcal F$-free} if there there is no $\str{F}\in \mathcal F$ with an
embedding $\str{F}\to \str{A}$. The class of all (finite and countably infinite) $\mathcal F$-free $L$-structures is denoted by $\Forb(\mathcal F)$.
With these definitions we can state a recent result:
\begin{theorem}[Zucker~\cite{zucker2020}]
	\label{thm:zucker}
	Let $L$ be a finite binary relational language, and $\mathcal F$ a finite set of finite irreducible $L$-structures.
	Then every $\Forb(\mathcal F)$-universal $L$-structure has finite big Ramsey degrees. (In other words, for every finite substructure $\str{A}$ of $\str{U}$
	there exists finite $D=D(\str{A})$ such that $\str{U}\longrightarrow (\str{U})^\str{A}_{k,D}$ for every $k>1$.)
\end{theorem}
This result can be seen as an infinitary variant of well known Ne\v set\v ril--R\"odl theorem (one of the fundamental results of structural Ramsey theory) which in our setting can be stated as follows:
\begin{theorem}[Ne\v set\v ril--R\"odl theorem~\cite{Nevsetvril1977,Nevsetvril1989}]
	\label{thm:NR}
	Let $L$ be a relational language, $\mathcal F$ a set of finite irreducible $L$-structures.
	Then for every finite $\str{A}\in \Forb(\mathcal F)$ there
	exists a finite integer $d=d(\str{A})$ such that for every finite $\str{B}\in \Forb(\mathcal F)$ and finite $k>0$ there exists a finite 
	$\str{C}\in \Forb(\mathcal F)$ satisfying $\str{C}\longrightarrow (\str{B})^\str{A}_{k,d}$.
\end{theorem}
To see the correspondence of Theorems~\ref{thm:zucker} and \ref{thm:NR} choose $\mathcal F$ as in Theorem~\ref{thm:zucker} 
and a finite $\mathcal F$-free $L$-structure $\str{A}$. By Theorem~\ref{thm:zucker} there is a finite
$D=D(\str{A})$ such that every $\Forb(\mathcal F)$-universal  $L$-structure $\str{U}$ satisfies $\str{U}\longrightarrow (\str{U})^\str{A}_{k,D}$ for every $k>0$.
By $\Forb(\mathcal F)$-universality of $\str{U}$ for every $\mathcal F$-free $L$-structure $\str{B}$ we have $\str{U}\longrightarrow (\str{B})^\str{A}_{k,D}$
and by compactness there exists a finite substructure $\str{C}$ of $\str{U}$ such that $\str{C}\longrightarrow (\str{B})^\str{A}_{k,D}$.
In general, $D(\str{A})$ (characterised precisely in~\cite{Balko2021exact}) differs from $d(\str{A})$ which corresponds to the number of non-isomorphic orderings of vertices of $\str{A}$.
However, the proof of Theorem~\ref{thm:zucker} can be adjusted to recover precise bounds on $d(\str{A})$.

Comparing  Theorems~\ref{thm:zucker} and~\ref{thm:NR}, it is natural to ask
 whether the assumptions about finiteness of $\mathcal F$, finiteness of the language $L$,  and relations
being only binary can be dropped from Theorem~\ref{thm:zucker}.  It is known that the first condition
can not be omitted:  Sauer~\cite{Sauer2002} has shown that there exist infinite families $\mathcal F$
of finite irreducible $L$-structures where $\Forb(\mathcal F)$-universal structures have infinite big Ramsey degrees of vertices.
This is true even for language $L$ containing only one binary relation (digraphs). The latter two conditions remain open.

Until recently, most bounds on big Ramsey degrees were for $L$-structures in binary languages only.
Techniques to give bounds on big Ramsey degrees of 3-uniform hypergraphs have been
announced in Eurocomb 2019~\cite{Hubickabigramsey} and published recently~\cite{Hubicka2020uniform}; they were later extended
to languages with arbitrary relational symbols~\cite{braunfeld2023big}. Extending links between the Hales--Jewett theorem~\cite{Hales1963}, Carlson--Simpson
theorem~\cite{carlson1984} and big Ramsey degrees established in~\cite{Hubicka2020CS}, a Ramsey-type theorem for trees with successor operations has been
introduced~\cite{Balko2023Sucessor} which extends to all known big Ramsey results on $L$-structures.
However, the following two problems remain open:
\begin{problem}[\cite{braunfeld2023big}]
	\label{prob:infinite}
	Does the random graph with countably many types of edges have finite big Ramsey degrees?
\end{problem}
While in the previous problem one can argue that infiniteness of the language makes it hard, there is a priori no reason to expect that the following class should cause troubles:
\begin{problem}
	\label{prob:badclique}
	Let $L=\{E,H\}$ be a language with one binary relation $E$ and one ternary relation $H$. Let $\str{F}$ be the
	$L$-structure where $F=\{0,1,2,3\}, R^F=\{(1,0),(1,2),(1,3)\}, H=\{(0,2,3)\}$. 
	Denote by $\mathcal K$ the class of all $L$-structures $\str{A}$ such that there is no monomorphism $\str{F}\to\str{A}$.
	Do $\mathcal K$-universal $L$-structures have finite big Ramsey degrees?
\end{problem}
These problems demonstrated unforeseen obstacles on giving a natural infinitary generalization of the Ne\v set\v ril--R\"odl theorem.
We give a new approach which avoids both these problems and which suggests perhaps the proper setting for big Ramsey degrees.

Finite structural Ramsey results are most often proved by refinements of the Ne\v set\v ril--R\"odl partite construction~\cite{Nevsetvril1989}.
This technique does not generalize to infinite structures due to essential use of backward induction.
Upper bounds on big Ramsey degrees are based on Ramsey-type theorems on trees (e.g. the Halper--L\"auchli theorem~\cite{Halpern1966}, Milliken's tree theorem~\cite{Milliken1979}, the Carlson--Simpson theorem~\cite{carlson1984}, and their various refinements~\cite{zucker2020,dobrinen2017universal,dobrinen2019ramsey}).
This proof structure may seem unexpected at first glance but is justified by the existence of unavoidable colourings (based on idea of Sierpi\'nski) which are constructed by assigning colors according to subtrees of the tree of 1-types (see e.g.~\cite{Balko2023,dobrinen2021ramsey}). The exact characterisations of big Ramsey degrees
can then be understood as an argument that this proof structure is in a very specific sense the only possible: the trees used to give upper bounds are also encoded in the precise characterisations of big Ramsey degrees.

We briefly review the construction of tree of 1-types.
Recall that a (model-theoretic) \emph{tree} is a partial order $(T,\leq)$ where the down-set of every $x\in T$ is a finite chain.
An \emph{enumerated $L$-structure} is simply an $L$-structure $\str{U}$ whose underlying set is the ordinal $|\str{U}|$. Fix  a countably infinite enumerated $L$-structure $\str{U}$.  Given vertices $u,v$ and an integer $n$ satisfying $\min(u,v)\geq n\geq 0$, we write $u\sim^\str{U}_n v$, and say that \emph{$u$ and $v$ are of the same (quantifier-free) type over $0,1,\ldots,n-1$}, if the $L$-structure induced by $\str{U}$ on $\{0,1,\ldots, n-1,u\}$ is identical to the $L$-structure induced by $\str{U}$ on $\{0,1,\ldots, n-1,v\}$ after renaming vertex $v$ to $u$. We write $[u]^\str{U}_n$ for the $\sim^\str{U}_n$-equivalence class of vertex $u$.
\begin{definition}[Tree of 1-types]
	Let $\str{U}$ be an infinite (relational) enumerated $L$-structure. Given $n< \omega$, write $\mathbb{T}_\str{U}(n) = \omega/\!\sim^{\str{U}}_n$. A (quantifier-free) \emph{1-type} is any member of the disjoint union $\mathbb{T}_\str{U}:=\bigsqcup_{n<\omega} \mathbb{T}_\str{U}(n)$. We turn $\mathbb{T}_\str{U}$ into a tree as follows. Given $x\in \mathbb{T}_\str{U}(m)$ and $y\in \mathbb{T}_\str{U}(n)$, we declare
	that $x\leq^{\mathbb T}_{\str{U}} y$ if and only if $m\leq n$ and $x\supseteq y$. 
\end{definition}
One can associate every vertex of $v\in \str{U}$ with its corresponding equivalence class in $\simeq^{\str{U}}_v$.
This way every substructure $\str{A}\subseteq \str{U}$ corresponds to a subset of nodes of the tree $\mathbb{T}_\str{U}$. 
Sierpi\'nski-like colourings can be then constructed by considering shapes of the meet closures of nodes corresponding to each given copy.

Every type in $x\in \omega/\sim^\str{U}_n$ can be described as an $L$-structure $\str{T}$ with vertex set $T=\{0,1,\ldots n-1,t\}$ such that for every $v\in x$ it holds that $L$-structure induced by $\str{U}$ on  $\{0,1,\ldots n-1,v\}$ is $\str{T}$ after renaming $t$ to $v$.
This is very useful in the setting where types originating from multiple enumerated $L$-structures are considered (which naturally happens in many big Ramsey proofs~\cite{Hubicka2020CS,Balko2023}).

The concept of the tree of 1-types was implicit in early proofs (such as in Devlin's thesis) and became explicit later.
The tree of 1-types itself is, however, not sufficient to give upper bounds on big Ramsey degrees for $L$-structures in languages containing
symbols of arity 3 and more.  Upper bounds in~\cite{Hubicka2020uniform} and~\cite{braunfeld2023big} are based on the product form of the
Milliken tree theorem which in turn suggests the following notion of a weak type.

For the rest of this note, fix a relational language $L$ containing a binary symbol $\leq$. For all $L$-structures, $\leq$ will always be a linear order on vertices which is either finite or of order-type $\omega$. This will describe the enumeration. All embeddings will be monotone.
\begin{definition}[Weak type]
	\label{def:wt}
	We denote by $L^f$ the language $L$ extended by unary function symbol $f$.
	An $L^f$-structure $\str{T}$ is a \emph{weak type} of \emph{level} $\ell$ if
	\begin{enumerate}
	\item $T=\{0,1,\ldots, \ell-1,t_0,t_1,\ldots\}$ where vertices $t_i$ are called \emph{type vertices}.
	\item\label{wt:2} For every $R\in L$ and $\vec{t}\in \rel{T}{}$ it holds that $\vec{t}\cap \{t_0,t_1,\ldots\}$ is a
		(possibly empty) initial segment of type vertices (i.e. set of the form $\{t_i : i\in k\}$ for some $k\in \omega$)
			and $\vec{t}\cap \{0,1,\ldots,\ell-1\}\neq \emptyset$.
	\item For every $i>0$ we put $\func{T}{}(t_i)=t_{i-1}$. We put $\func{T}{}(t_0)=t_0$ and $\func{T}{}$ is undefined otherwise.
	\end{enumerate}
\end{definition}
Weak types thus give less information than standard model-theoretic $k$-types~\cite{Hodges1993}.
Function $f$ is added to type vertices to distinguish them from normal vertices.
This will be useful in later constructions.
Notice that while technically weak type has infinitely many types vertices, thanks to condition~\ref{wt:2} of Definition~\ref{def:wt},
if the language $L$ contains no relations of arity $r+1$ or more, vertices $t_{r-1},t_{r},\ldots$ will be isolated.  In particular:
\begin{observation}
	\label{obs:binary}
If $L$ contains only unary and binary symbols then there is one-to-one correspondence between 1-types and weak types
because only type vertex $t_0$ carries interesting structure.
\end{observation}

1-types describes one vertex extensions of an initial part of the enumerated $L$-structure. The weak-type equivalent of this is the following:

\begin{definition}[Weak type of a tuple]
	Let $\str{A}$ be an enumerated $L$-structure, $\str{T}$
	a weak type of level $\ell\in A\subseteq \omega$ and $\vec{a}=(a_0,a_1,\ldots, a_{k-1})$ an increasing tuple  of vertices from $A\setminus \ell$. We say that
	\emph{$\vec{a}$ has type $\str{T}$} on level $\ell$ if the function $h\colon T\to A$
	given by:
	$$
	h(x)=\begin{cases}
		x & \hbox{if $x\in \ell$},\\
		a_i & \hbox{if $x= t_i$ for some $i<k$}
	\end{cases}
	$$
	has the property that for every $R\in L$ and $\vec{b}$ a tuple of vertices in $\{0,1,\ldots,\ell-1,t_0,t_1,\ldots, t_{k-1}\}$ such that $\vec{b}\cap \{t_0,t_1,\ldots\}$ is an initial segment
	of type vertices and $\vec{b}\cap \{0,1,\ldots,\ell-1\}\neq \emptyset$ it holds that  $\vec{b}\in \rel{T}{}\iff h(\vec{b})\in \rel{A}{}$.
\end{definition}
\begin{definition}[Tree of weak types]
	Given an enumerated $L$-structure $\str{U}$, its \emph{tree of weak types} consists of all $L^f$-structures $\str{T}$ that are weak types of some tuple of $\str{U}$ on some level $\ell\in U$ ordered by $\subseteq$.
\end{definition}
Given an enumerated $L$-structure $\str{A}$ and a weak type $\str{T}$, we say that $\str{T}$ \emph{extends} $\str{A}$ if $\str{T}\setminus\{t_0,t_1,\ldots\}=\str{A}$.
Given two types $\str{T}$ and $\str{T}'$ that extend $\str{A}$, and $n\geq 0$, we say that $\str{T}$ and $\str{T}'$ agree as $n$-types if $\str{T}\restriction (A\cup\{t_0,t_1,\ldots t_{n-1}\})=\str{T}'\restriction (A\cup\{t_0,t_1,\ldots t_{n-1}\})$.

A standard technique for proving infinite Ramsey-type theorems is to work with finite approximations of the embeddings considered.  See e.g. Todorcevic's axiomatization of Ramsey spaces~\cite{todorcevic2010introduction}.
Initial approximations of our embeddings will be described as follows:
\begin{definition}[Structure with types]
	Given a finite enumerated $L$-structure $\str{A}$, we denote by $\str{A}^+$ the $L$-structure created from the disjoint union of all weak types extending $\str{A}$ by
	\begin{enumerate}
		\item identifying all copies of $\str{A}$, and,
		\item identifying the copy of vertex $t_i$ of weak type $\str{T}$ and with the copy of $t_i$ of weak type $\str{T}'$ whenever $\str{T}$ and $\str{T}'$ agree as $i+1$ types.
	\end{enumerate}
\end{definition}
Observe that thanks to the function $f$ added to weak types, for any two $L$-structures with types $\str{A}^+$ and $\str{B}^+$, every embedding $h\colon\str{A}^+\to\str{B}^+$ is
also a map from weak types of $\str{A}$ on level $\lvert A\rvert$ to weak types of $\str{B}$ of level $\lvert B\rvert$.

Given an $L$-structure $\str{A}$ and a vertex $v$, we denote by $\str{A}({<}v)$ the $L$-structure induced by $\str{A}$ on $\{a\in A; a<v\}$ and call it the \emph{initial segment} of $\str{A}$.
The key notion for our approach is to restrict attention to embedding which behave well with respect to weak types. That is, for every initial segment of the $L$-structure, the rest of
the embedding can be summarized via embedding of weak types  extending the initial segments.
\begin{definition}[Type-respecting embeddings of  $L$-structures]
		Given  enumerated $L$-structures $\str{A}$ and $\str{B}$ and an embedding $h\colon\str{A}\to \str{B}$, we say that $h$ is \emph{type-respecting}
		if for every $v\in A$ there exists an embedding $h^v\colon\str{A}({<}v)^+\to \str{B}({<}h(v))^+$ such that the weak types of tuples in $\str{B}$
		on level $h(v)$ consisting only of vertices of $h[A]$ are all in the image $h^v[\str{A}]$.
\end{definition}
\begin{definition}[$\mathcal K$-type-respecting embeddings of initial segments]
Let $\str{A}$ and $\str{B}$ be two finite enumerated $L$-structures.
Embedding $h\colon\str{A}^+\to \str{B}^+$ is \emph{type-respecting} if for every (possibly infinite) $L$-structure $\str{A}'$ with initial segment $\str{A}$ there exists an $L$-structure $\str{B}'$ with initial segment $\str{B}$ and a type-respecting embedding $g\colon \str{A}\to\str{B}$ 
	finitely approximated by $h$. That is $g\restriction A=h\restriction A$ and every weak type in $\str{B}'$ of a tuple consisting of vertices of $g[A]$ of level $g(\mathop{\mathrm{max}} A)$ is in $h[A^+]$.

	Given class $\mathcal K$ of $L$-structures we say that  $h\colon\str{A}^+\to \str{B}'^+$ is \emph{$\mathcal K$-type-respecting} if for every $L$-structure $\str{A}'\in \mathcal K$ with initial segment $\str{A}$ there exists an structure $\str{B}'\in \mathcal K$ with initial segment $\str{B}$ and a type-respecting embedding $g\colon\str{A}\to\str{B}$ finitely approximated by $h$.
\end{definition}
	\begin{definition}[Type-respecting amalgamation property]
		\label{def:typeamalg}
		Let $\mathcal K$ be a hereditary class of enumerated $L$-structures. We say that $\mathcal K$ has \emph{type-respecting amalgamation property} if 
		given three finite enumerated $L$-structures $\str{A}$, $\str{B}$, $\str{B}'\in \mathcal K$ such that $B'\setminus B=\{\mathop{\mathrm{max}} B'\}$ and $\str{B}'\restriction \str{B}=\str{B}$,
		two $\mathcal K$-type-respecting embeddings $f\colon\str{A}^+\to \str{B}^+$, $f'\colon\str{A}^+\to \str{B}'^+$ and a type-respecting (but not necessarily $\mathcal K$-type-respecting) embedding $g\colon\str{B}^+\to \str{B}'^+$ such that $g\restriction B$ is the identity and $g\circ f=f'$,
		there exists a $\mathcal K$-type-respecting embedding $g'\colon\str{B}^+\to \str{B}'^+$ such that
		$g'\circ f=f'$ and $g'\restriction B=\mathrm{Id}$.
	\end{definition}

Given a class of $L$-structures $\mathcal K$, finite $\str{A}\in \mathcal K$ and $\str{B}\in \mathcal K$, we denote by ${\str{B}\choose \str{A}}^{\mathcal K}$ the set
of all $\mathcal K$-type-respecting embeddings $\str{A}^+\to \str{B}'^+$ for $\str B'$ an initial segment of $\str B$. We write $\str{C}\longrightarrow^{\mathcal K} (\str{B})^\str{A}_{k,l}$ to denote the following statement:
for every colouring $\chi$ of ${\str{C}\choose\str{A}}^{\mathcal K}$ with $k$ colours, there exists
a type-respecting embedding $f\colon\str{B}\to\str{C}$ such that $\chi$ does not take more than $l$ values on ${f(\str{B})\choose \str{A}}^\mathcal K$.
For a countably infinite $L$-structure $\str{B}$ and its finite suborder $\str{A}$, the \emph{big Ramsey degree of $\str{A}$ in $\mathcal K$-type-respecting embeddings} of $\str{A}$ in $\str{B}$ is
the least number $D\in \mathbb N\cup \{\infty\}$ such that $\str{B}\longrightarrow^{\mathcal K} (\str{B})^\str{A}_{k,D}$ for every $k\in \mathbb N$.

For type-respecting embeddings we can prove the Ramsey property in the full generality (showing that, in this situation, Problem~\ref{prob:badclique} is not a problem).

	\begin{theorem}
		\label{thm:main}
		Let $L$ be a finite relational language.
		Let $\mathcal F$ be a finite family of finite irreducible enumerated $L$-structures. Denote by $\mathcal K_\mathcal F$ the class of all finite or countably-infinite enumerated $L$-structures $\str{A}$ where $\leq_\str{A}$ is either finite or of order-type $\omega$  such that for every $\str{F}\in \mathcal F$ there no embedding $\str{F}\to \str{A}$. Assume that $\mathcal K_\mathcal F$ has the type-respecting amalgamation property. 
		Then for every universal $L$-structure $\str{U}\in \mathcal K_\mathcal F$ and every finite $\str{A}\in \mathcal K_\mathcal F$ there is a finite $D=D(\str{A})$ such that $\str{U}\longrightarrow^\mathcal{K} (\str{U})^{\str{A}}_{k,D}$ for every $k\in \mathbb N$.
	\end{theorem}

We show the following:
\begin{prop}
	\label{prop1}
Let $L$ be a finite language consisting of binary and unary relational symbols only.  Let $\mathcal F$ be a finite family of enumerated irreducible $L$-structures.
Then $\mathcal K_\mathcal F$ has the type-respecting amalgamation property.
	Moreover, Theorem~\ref{thm:main} implies Theorem~\ref{thm:zucker}.
\end{prop}
\begin{proof}
	Fix $L$, $\mathcal F$ and $\mathcal K_\mathcal F$.  Let $\str{A}$, $\str{B}$, $\str{B}'\in \mathcal K_\mathcal F$,
	$f\colon\str{A}^+\to \str{B}^+$, $f'\colon\str{A}^+\to \str{B}'^+$ and  $g\colon\str{B}^+\to \str{B}'^+$ be as in Definition~\ref{def:typeamalg}.
	By Observation~\ref{obs:binary}, in order to specify $g'$, it is only necessary to give, for every weak type $\str{T}$ extending $\str{A}$, an image of its type vertex $t_0\in T$.
	Let $t'\in B^+$ be a vertex corresponding to $t_0$.  We consider two cases. (1) If $t'\in f[A^+]$ then we put $g'(t')=g(t')$. (2) If $t'\notin f[A^+]$ we put $g'(t')=t''$ where
	$t''$ is the only possible image of $t'$ such that there is no relational symbol $R\in L$ such that $\nbrel{\str{B}^+}{}$ contains a tuple with both $t''$ and $\mathop{\mathrm{max}} B'$.

	We verify that $g'$ is $\mathcal K_\mathcal F$-type-respecting. Choose $\str{A}'\in \mathcal K_\mathcal F$ with initial segment $\str{B}$.
	Construct $\str{A}''$ from $\str{A}$ by inserting new vertex $v$ after $\mathop{\mathrm{max}} B$ and extending $\leq_{\str{A}''}$ accordingly. Then add all tuples to relations necessary to make $\str{B}'$ the initial segment of $\str{A}''$.
	Finally, for every $\rel{}{}\in L$ and $u\in A'$ with $u>v$, put $(u,v)\in \nbrel{\str{A}''}{}$ if and only if $(g'(t),u)\in \nbrel{\str{B}^+}{}$ where is $t$ is the type vertex of $\str{B}$ corresponding to the
	type of $u$ in $\str{A}'$. Add tuples $(v,u)\in \nbrel{\str{A}''}{}$ analogously.

	It remains to verify that $\str{A}''\in \mathcal K_\mathcal F$. Assume to the contrary that there is $\str{F}\in \mathcal F$ and embedding $e\colon\str{F}\to\str{A}''$. Because $\str{A}'\in \mathcal K_\mathcal F$, clearly $v\in e[F]$.
	Because $\str{B}'\in \mathcal K_\mathcal F$ we also know that $e[F]$ contains vertices of $\str{A}''\setminus str{B}'$. Since $\str{F}$ is irreducible, all such vertices must have types created by condition (1) above.
	This is a contradiction with $f'$ being $\mathcal K_\mathcal F$-type-respecting.

	To see the moreover part we have to construct a universal $\str{U}$
	which is a substructure of some $\str{U}'\in \mathcal K_\mathcal F$
	with the property that for every $n\in \mathbb N$ there exists $N\in \mathbb N$ such that for every $\str{A}\in \mathcal K_\mathcal
	F$ with $n$ vertices and every embedding $e\colon\str{A}\to \str{U}'$ there exist a structure $\str{E}\in \mathcal K_\mathcal F$ 
	(called an \emph{envelope}) with at most $N$ vertices and a $\mathcal K_\mathcal F$-type-respecting embedding
	$h\colon\str{E}\to\str{U}$ such that $e[A]\subseteq h[E]$.  This follows from Section~4 of~\cite{zucker2020}, because $\mathcal K$-type-respecting embeddings in this setup are precisely the aged embeddings used by Zucker in~\cite{zucker2020}.
\end{proof}
\begin{prop}
	Let $L'=\{E,H,\leq\}$ and let $\str{F}'$ be an $L'$-structure created by expanding the $L$-structure $\str{F}$ from Problem~\ref{prob:badclique} by the natural
	order of vertices. Denote by $\mathcal K$ the class of all enumerated $L'$-structures $\str{A}$ for which there is no monomorphism $\str{F}\to\str{A}$.
	The class $\mathcal K$ has no type-respecting amalgamation property.
\end{prop}
\begin{proof}
	We give an explicit failure of type-respecting amalgamation showing that the use of Observation~\ref{obs:binary} in the previous proof is essential. Let $\str{A}$ be the empty $L'$-structure, $\str{B}$ be $L'$-structure with $B=\{0\}, E_\str{B}=H_\str{B}=\emptyset$ and $\str{B}'$ be $L'$-structure with $B'=\{0,1\}$, $E_{\str{B}'}=\{(0,1)\}$, $H_{\str{B}'}=\emptyset$.
	Let $\str{T}_\str{A}$ be the unique weak type extending $\str{A}$.  Let $\str{T}_\str{B}$ be weak type extending $\str{B}$ with $E_{\str{T}_\str{B}}=H_{\str{T}_\str{B}}=\emptyset$ and $\str{T}'_\str{B}$ weak type extending $\str{B}$ with $E_{\str{T}'_\str{B}}=\emptyset$ and $H_{\str{T}'_\str{B}}=\{(0,t_0,t_1)\}$. Notice that $\str{T}_\str{B}$ and $\str{T}'_\str{B}$ agree as 1-types and thus in $\str{B}^+$ their vertices $t_0$ are identified. Finally, let $\str{T}_{\str{B}'}$ and $\str{T}'_{\str{B}'}$ be weak types extending $\str{B}'$ with $E_{\str{T}_{\str{B}'}}=H_{\str{T}'_{\str{B}'}}=\{(0,1),(1,t_0)\}$, $H_{\str{T}_{\str{B}'}}=\emptyset$, $H_{\str{T}'_{\str{B}'}}=\{(0,t_0,t_1)\}$. Again $\str{T}_{\str{B}}$ and $\str{T}_{\str{B}'}$ agree as 1-types.
	Now let $f\colon\str{A}^+\to \str{B}^+$ map $\str{T}_\str{A}$ to $\str{T}_\str{B}$ and $f^+\colon\str{A}^+\to \str{B}^+$ map  $\str{T}_\str{A}$ to $\str{T}_{\str{B}'}$. It is easy to check that these are $\mathcal K$-type-respecting. $g\colon\str{B}^+\to \str{B}'^+$ can be constructed to be type-respecting by mapping type $\str{T}_\str{B}$ to $\str{T}_{\str{B}'}$ and $\str{T}'_\str{B}$ to $\str{T}'_{\str{B}'}$. However there is no $\mathcal K$-type-respecting $g'\colon\str{B}^+\to\str{B}'^+$. To see that, observe that any image of $\str{T}'_\str{B}$ must agree as 1-type with $\str{T}_{\str{B}'}$ and consider $\str{A'}$ with $A'=\{0,1,2\}$ and $H_{\str{A}'}=\{(0,1,2)\}$. $\str{A}$ is an initial segment of $\str{A}'$ and there is no way to extend $g'$ to a $\mathcal K$-type-respecting embedding of $\str{A}'$ to some $L$-structure in $\mathcal K$ since it will always add vertex $v$ after vertex $0$ of $\str{A}$ in a way that there is a monomorphism from $\str{F}$ to  $\{0,v,1,2\}$.
\end{proof}
We conjecture that the answers to Problems~\ref{prob:infinite} and~\ref{prob:badclique} are in fact negative. It is possible that by concentrating on type-respecting embeddings, the study of big Ramsey degrees can find a proper setting.

\paragraph{\bf{Acknowledgement}}
First six authors are supported by the project 21-10775S of  the  Czech  Science Foundation (GA\v CR). This article is part of a project that has received funding from the European Research Council (ERC) under the European Union's Horizon 2020 research and innovation programme (grant agreement No 810115).
\bibliographystyle{plain}
\bibliography{ramsey.bib}
\end{document}